\documentclass[12pt, reqno]{amsart}
\usepackage{amssymb, amsmath, amsthm, cancel, hyperref, color, enumerate, verbatim, mathrsfs}
\usepackage{mathtools}

\usepackage{ulem, cancel}  % for strikeout purposes 

\usepackage[margin=1in]{geometry}
\usepackage{tikz} % for the diagrams
\usetikzlibrary{arrows.meta} % for the diagrams
\usetikzlibrary{decorations.markings} % for the diagrams

\newtheorem{thm}{Theorem}[section]
\newtheorem{lem}[thm]{Lemma}
\newtheorem{prop}[thm]{Proposition}
\newtheorem{cor}[thm]{Corollary}
\newtheorem{conjecture}[thm]{Conjecture}

\theoremstyle{definition}
\newtheorem{defn}[thm]{Definition}

\theoremstyle{remark}
\newtheorem*{rem}{Remark}

\newcommand{\CC}{\mathbb{C}}
\newcommand{\la}{\lambda} % lambda
\newcommand{\ZZ}{\mathbb{Z}}

\newcommand{\OO}{\operatorname*{\Omega}_{\geq}}

\title{$d$-fold partition diamonds}
\author{Dalen Dockery}
\address{Department of Mathematics, University of Tennessee, Knoxville, TN 37996, USA}
\email{ddocker5@vols.utk.edu}

\author{Marie Jameson}
\address{Department of Mathematics, University of Tennessee, Knoxville, TN 37996, USA}
\email{mjameso2@utk.edu}

\author{James A. Sellers}
\address{Department of Mathematics and Statistics, University of Minnesota Duluth, Duluth, MN 55812, USA}
\email{jsellers@d.umn.edu}

\author{Samuel Wilson}
\address{Department of Mathematics, University of Tennessee, Knoxville, TN 37996, USA}
\email{swils133@vols.utk.edu}

\begin{document}

\maketitle

\begin{abstract}
In this work we introduce new combinatorial objects called $d$--fold partition diamonds, which generalize both the classical partition function and the partition diamonds of Andrews, Paule and Riese, and we set $r_d(n)$ to be their counting function. We also consider the Schmidt type $d$--fold partition diamonds, which have counting function $s_d(n).$ Using partition analysis, we then find the generating function for both, and connect the generating functions $\sum_{n= 0}^\infty s_d(n)q^n$ to Eulerian polynomials.  This allows us to develop elementary proofs of infinitely many Ramanujan--like congruences satisfied by $s_d(n)$ for various values of $d$, including the following family:  for all $d\geq 1$ and all $n\geq 0,$ $s_d(2n+1) \equiv 0 \pmod{2^d}.$  
\end{abstract}

% partitions, q-series, Schmidt type partitions, partition diamonds, MacMahon's partition analysis, Ramanujan congruences
% MSC code: 11P82 (with 11P83 as a possible secondary one)

\section{Introduction and Statement of Results}
\label{sec:intro}

In 2001, George Andrews, Peter Paule, and Axel Riese \cite{AndrewsPauleRiese01b} defined plane partition diamonds, which are partitions whose parts are non-negative integers $a_i$ and $b_i$ which are placed at the nodes of the graph given in Figure \ref{fig:2fold}, where each directed edge indicates $\geq.$ They found that the generating function for such partitions is given by \cite[Corollary 2.1]{AndrewsPauleRiese01b}
\begin{equation}\label{eq:r2}
\prod_{n=1}^\infty \frac{1+q^{3n-1}}{1-q^n}.
\end{equation}
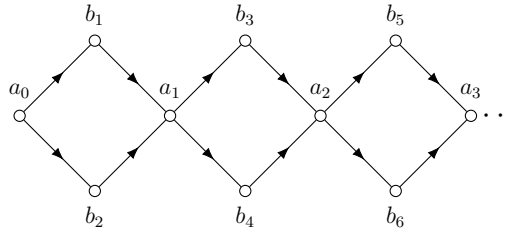
\begin{figure}
\centering
\begin{tikzpicture}
\begin{scope}[every node/.style={circle,draw,scale=0.75,inner sep=2pt}]
    \node[label={$a_0$}] (a0) at (0,0) {};
    \node[label={$b_1$}] (b1) at (1,1) {};
    \node[label=below:{$b_2$}] (b2) at (1,-1) {};
    \node[label={$a_1$}] (a1) at (2,0) {};
    \node[label={$b_3$}] (b3) at (3,1) {};
    \node[label=below:{$b_4$}] (b4) at (3,-1) {};
    \node[label={$a_2$}] (a2) at (4,0) {};
    \node[label={$b_5$}] (b5) at (5,1) {};
    \node[label=below:{$b_6$}] (b6) at (5,-1) {};
    \node[label={$a_3$}] (a3) at (6,0) {};
\end{scope}

\begin{scope}[>={Stealth[black]}, decoration={
    markings,
    mark=at position 0.6 with {\arrow{Latex}}}
    ] 
    \draw[postaction={decorate}] (a0) -- (b1);
    \draw[postaction={decorate}] (a0) -- (b2);
    \draw[postaction={decorate}] (b1) -- (a1);
    \draw[postaction={decorate}] (b2) -- (a1);
    \draw[postaction={decorate}] (a1) -- (b3);
    \draw[postaction={decorate}] (a1) -- (b4);
    \draw[postaction={decorate}] (b3) -- (a2);
    \draw[postaction={decorate}] (b4) -- (a2);
    \draw[postaction={decorate}] (a2) -- (b5);
    \draw[postaction={decorate}] (a2) -- (b6);
    \draw[postaction={decorate}] (b5) -- (a3);
    \draw[postaction={decorate}] (b6) -- (a3) node[right] {$\cdots$};
\end{scope}
\end{tikzpicture}
\caption{plane partition diamonds ($d=2$)}
\label{fig:2fold}
\end{figure}
In this paper, we aim to generalize this construction. For any positive integer $d$, we place $d$ nodes between any two nodes $a_k$ and $a_{k+1}$, so that we have the classical partition function $p(n)$ when $d=1$ and we have the plane partition diamonds of Andrews, Paule, and Riese when $d=2.$

More precisely, we define $d$-fold partition diamonds to be partitions whose parts are non-negative integers $a_i$ and $b_{i,j}$ which are placed at the nodes of the graph given in Figure \ref{fig:dfold}, where each directed edge indicates $\geq.$

\begin{figure}
\centering
\begin{tikzpicture}
\begin{scope}[every node/.style={circle,draw,scale=0.75,inner sep=2pt}]
\foreach \i/\k in {0/0,2/1,4/2,6/3}
    {\node[label={$a_{\k}$}] (a\k) at (\i,0) {};}
\foreach \i in {1,2,3}
    {\node[] (b1\i) at (1,-0.5*\i+1.5) {};
     \node[] (b2\i) at (3,-0.5*\i+1.5) {};
     \node[] (b3\i) at (5,-0.5*\i+1.5) {};
     \node[] (b\i d) at (2*\i-1,-1) {};}
\foreach \i in {1,2,3}
    {\node also [label=above:$b_{\i, 1}$] (b\i 1);
     \node also [label=below:$b_{\i, d}$] (b\i d);}
\end{scope}

\begin{scope}[>={Stealth[black]}, decoration={
    markings,
    mark=at position 0.6 with {\arrow{Latex}}}]
\foreach \i/\k in {0/1,1/2,2/3}
    {\foreach \j in {1,2,3,d}
    {\draw[postaction=decorate] (a\i) -- (b\k\j);
     \draw[postaction=decorate] (b\k\j) -- (a\k);
    }}
\foreach \i in {1,3,5}
    {\node[] at (\i,-0.33) {$\vdots$};}
\node[right] at (6,0) {$\cdots$};
\end{scope}
\end{tikzpicture}
\caption{$d$-fold partition diamonds}
\label{fig:dfold}
\end{figure}
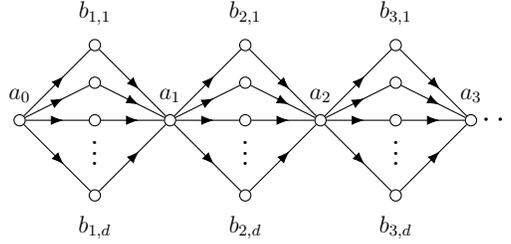

Letting $r_d(n)$ denote the number of $d$-fold partition diamonds of $n,$ we then have that
\begin{align*}
\sum_{n=0}^\infty r_1(n)q^n &= \prod_{n=1}^\infty \frac{1}{1-q^n},\\
\sum_{n=0}^\infty r_2(n)q^n &= \prod_{n=1}^\infty \frac{1+q^{3n-1}}{1-q^n},\\
\sum_{n=0}^\infty r_3(n)q^n &= \prod_{n=1}^\infty \frac{1 + 2q^{4n-2}(1+q) + q^{8n-3}}{1-q^n}.
\end{align*}
Note that this corresponds to the classical partition function when $d=1$ and to \eqref{eq:r2} when $d=2.$ More generally, we have the following result.
\begin{thm}\label{thm:genfcnforall}
The generating function for the number of $d$-fold partition diamonds is given by
\[\sum_{n=0}^\infty r_d(n)q^n =\prod\limits_{n=1}^{\infty} \frac{F_{d}(q^{(n-1)(d+1)+1},q)}{1-q^n} ,\]
where $F_d$ is a polynomial that is defined recursively in Lemma \ref{lem:basecase}.
\end{thm}

In fact, we can also prove an analogous result regarding Schmidt type $d$-fold partition diamonds, which are a variant obtained by summing only the links $a_0 + a_1 + \cdots + a_k$ (i.e., omitting the $b_i$'s when expressing $n$ as a sum) as in \cite{AndrewsPaule22}. Letting $s_d(n)$ denote the number of Schmidt type $d$-fold partition diamonds of $n$, we have that
\begin{align*}
\sum_{n=0}^\infty s_1(n)q^n &= \prod_{n=1}^\infty \frac{1}{(1-q^n)^2},\\
\sum_{n=0}^\infty s_2(n)q^n &= \prod_{n=1}^\infty \frac{1+q^n}{(1-q^n)^3},\\
\sum_{n=0}^\infty s_3(n)q^n &= \prod_{n=1}^\infty \frac{1+4q^n+q^{2n}}{(1-q^n)^4}.
\end{align*}
Note that the generating functions for $d=1$ and $d=2$ given above were known previously \cite[Theorems 2 and 4]{AndrewsPaule22}. More generally, we have the following result.

\begin{thm}\label{thm:genfcnforlinks}
The generating function for the number of Schmidt type $d$-fold diamond partitions is given by 
\[\sum_{n=0}^\infty s_d(n)q^n = \prod_{n=1}^\infty \frac{A_{d}(q^n)}{(1-q^{n})^{d+1}}.\]
\end{thm}

Here, $A_d(q)$ denotes the $d$th Eulerian polynomial, which is defined by $A_0(q)=1$ and, for all $d\geq 1,$ 
\begin{equation} \label{eulerian_poly_recur}
A_d(q) = (1+(d-1)q)A_{d-1}(q)+q(1-q)A'_{d-1}(q).
\end{equation}

Both Theorems \ref{thm:genfcnforall} and \ref{thm:genfcnforlinks} arise as special cases of a more technical theorem, Theorem \ref{thm:mainthm}, which will be proved using MacMahon's partition analysis. 

In Section 2, we will describe the MacMahon operator and use it to understand a generating function for $d$-fold partition diamonds; this will allow us to state our main technical theorem, Theorem \ref{thm:mainthm}. In Section \ref{sec:prelimresults}, we will prove some preliminary results needed to prove Theorem \ref{thm:mainthm}, which we do in Section \ref{sec:mainthm}. Finally, in Section \ref{sec:congproofs}, we prove several infinite families of Ramanujan--like congruences satisfied by the functions $s_d(n)$ in elementary fashion. For example, we prove that for all $d\geq 1$ and all $n\geq 0,$ \[s_d(2n+1) \equiv 0 \pmod{2^d}.\]

\section{Background and Notation} \label{sec:background}

In order to describe the generating functions that arise in studying $d$-fold diamond partitions, we start by introducing the MacMahon operator $\Omega_{\geq},$ which is the primary tool for computing these generating functions. The discussion below is inspired by the landmark series of papers on the subject by Andrews, Paule, and others \cite{AndrewsPauleRiese01a, AndrewsPauleRiese01b, AndrewsPaule22}.

\subsection{The MacMahon Operator}

First we define MacMahon's Omega operator $\OO.$

\begin{defn}
The operator $\OO$ is given by
\[\OO \sum_{s_1=-\infty}^\infty \cdots \sum_{s_r=-\infty}^\infty A_{s_1,\ldots,s_r} \la_1^{s_1} \cdots \la_r^{s_r} \coloneqq \sum_{s_1=0}^\infty \cdots \sum_{s_r=0}^\infty A_{s_1,\ldots,s_r},\]
        where the domain of the $A_{s_1,\ldots,s_r}$ is the field of rational functions over $\CC$ in several complex variables and the $\lambda_i$ are restricted to a neighborhood of the circle $|\la_i|=1.$ In addition, the $A_{s_1,\ldots,s_r}$ are required to be such that any of the series involved is absolutely convergent within the domain of the definition of $A_{s_1,\ldots,s_r}.$
\end{defn}

We will always have $\OO$ operate on variables denoted by letters $\la, \mu$ from the Greek alphabet (so that letters from the Latin alphabet will be unaffected by $\OO$).

The benefit of using the MacMahon operator comes from our ability to compute it quickly using elimination formulae. MacMahon \cite[pp. 102-103]{MacMahon20} gave a list of some of these, e.g.,
\begin{equation} \label{eq:MacMahon1}
\OO \frac{1}{(1-\la x)(1-\la^{-1} y)} = \frac{1}{(1-x)(1-xy)}.
\end{equation}
This formula can be generalized as follows. 

\begin{prop}\label{prop:MacMahon}
For $d$ a positive integer and $j$ an integer, we have that
\begin{multline*}
    \OO \frac{\lambda^j}{(1-\la x_1)\cdots(1-\la x_d)(1-\la^{-1}y)} = \frac{1}{(1-y)}\left[\frac{1}{(1-x_1)(1-x_2)\cdots(1-x_d)}\right. \\ \left.-\frac{y^{j+1}}{(1-x_1y)(1-x_2y)\cdots(1-x_dy)}\right] .
\end{multline*}
\end{prop}

\begin{proof}
Expanding the left hand side using geometric series and applying the definition of $\OO$, we obtain
\begin{align*}
\OO \frac{\lambda^j}{(1-\lambda x_1)\cdots(1-\lambda x_d)(1-\lambda^{-1}y)} &= \OO\sum_{a_1,\ldots,a_{d+1}\geq 0}\lambda^{j+a_1+\cdots+a_d-a_{d+1}}x_1^{a_1}\cdots x_d^{a_d}y^{a_{d+1}} \\
&= \OO\sum_{\substack{a_1,\ldots,a_d\geq 0 \\ 0\, \leq\, k \,\leq\, {a_1+\cdots+a_d+j}}} \lambda^kx_1^{a_1}\cdots x_d^{a_d}y^{a_1+\cdots+a_d+j-k} \\
&= \sum_{a_1,\ldots,a_d \geq 0} x_1^{a_1}\cdots x_d^{a_d} \sum_{k=0}^{a_1+\cdots+a_{d}+j} y^k
\\
&=
\sum_{a_1,\ldots,a_d\geq 0} x_1^{a_1}\cdots x_d^{a_d} \left[\frac{1-y^{a_1+\cdots+a_d+j+1} }{1-y} \right],
\end{align*}
and the result follows by simplifying the geometric series.
\end{proof}

\subsection{Generating Functions for $d$-fold Partition Diamonds} \label{sec:notation}

Following the work of Andrews, Paule, and Riese in \cite{AndrewsPauleRiese01b}, for integers $d, n\geq 1$ we set
\[D_{d,n} \coloneqq D_{d,n}(q_0, q_1, \ldots, q_n; w) \coloneqq \sum q_{0}^{a_0}\cdots q_{n}^{a_n} w^{\sum_{j,k} b_{j,k}}\]
where the (outer) sum for $D_{d,n}$ ranges over all non-negative integers $a_i$ and $b_{j,k}$ which satisfy the inequalities encoded by Figure \ref{fig:dfoldlengthn}. Thus, $D_{d,n}$ is the generating function for all $d$-fold partition diamonds of fixed length $n$, which contain nodes $a_i$ and $b_{j,k}$ for $0\leq i\leq n$, $1\leq j\leq n$, and $1\leq k\leq d$; each part (i.e., vertex of the graph) corresponds to some $q_i$ or $w$ in this generating function.

\begin{figure}
\centering
\begin{tikzpicture}
\begin{scope}[every node/.style={circle,draw,scale=0.75,inner sep=2pt}]
\foreach \i/\k in {0/0,2/1,4/2,5/n-1,7/n}
    {\node[] (a\k) at (\i,0) {};}
\foreach \i in {1,2,3}
    {\node[] (b1\i) at (1,-0.5*\i+1.5) {};
     \node[] (b2\i) at (3,-0.5*\i+1.5) {};
     \node[] (bn\i) at (6,-0.5*\i+1.5) {};}
\foreach \i/\k in {1/1,2/2,3.5/n}
    { \node[] (b\k d) at (2*\i-1,-1) {};}
\foreach \i in {1,2,n}
    {\node also [label=above:$b_{\i, 1}$] (b\i 1);
     \node also [label=below:$b_{\i, d}$] (b\i d);}
\foreach \k in {0,1,2,n}{
\node also [label={[label distance= 0.2cm]:$a_{\k}$}] (a\k);
}
\foreach \k in {n-1}{\node also [label={[label distance= 0.01cm]90:$a_{n-1}$}] (a\k);}
\end{scope}
\begin{scope}[>={Stealth[black]}, decoration={
    markings,
    mark=at position 0.6 with {\arrow{Latex}}}]
\foreach \i/\k in {0/1,1/2,n-1/n}
    {\foreach \j in {1,2,3,d}
    {\draw[postaction=decorate] (a\i) -- (b\k\j);
     \draw[postaction=decorate] (b\k\j) -- (a\k);
    }}
\foreach \i in {1,3,6}
    {\node[] at (\i,-0.33) {$\vdots$};}
\node[] at (4.5,0) {$\cdots$};   
\end{scope}
\end{tikzpicture}
\caption{$d$-fold partition diamond of length n.}
\label{fig:dfoldlengthn}
\end{figure}
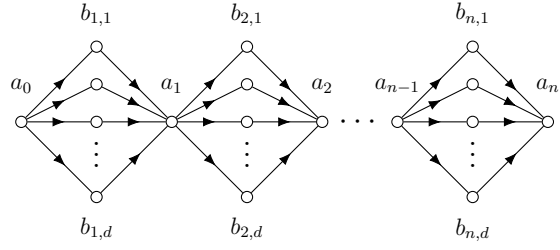
Using this notation, we are now ready to state our main technical theorem.
\begin{thm} \label{thm:mainthm}
For $d\geq 1$ and $n\geq 1,$
\[D_{d,n}(q_0, \ldots, q_n; w) = \left(\prod_{k=0}^{n-1} \frac{F_d(Q_kw^{dk},w)}{(1-Q_kw^{dk})(1-Q_kw^{dk+1})\cdots (1-Q_kw^{dk+d})}\right) \frac{1}{1-Q_nw^{dn}},\]
where $F_d$ is the polynomial that is defined recursively in Lemma \ref{lem:basecase}, and $Q_k \coloneqq q_0q_1\dots q_k$ for all $k\geq 0.$
\end{thm}

\subsection{The Crude Form of the Generating Function}

Here, we continue to generalize the objects described in \cite{AndrewsPauleRiese01b} as we fix some more notation to describe the crude form of the generating function $D_{d,n}.$ For integers $d, n\geq 1$ we set
\[\Lambda_n \coloneqq \Lambda_{d,n} \coloneqq \left(\la_{n,1}^{a_{n-1}-b_{n,1}} \la_{n,2}^{a_{n-1}-b_{n,2}} \cdots \la_{n,d}^{a_{n-1}-b_{n,d}}\right) \left(\mu_{n,1}^{b_{n,1}-a_n} \mu_{n,2}^{b_{n,2}-a_n} \cdots \mu_{n,d}^{b_{n,d}-a_n}\right).\]
Note here that $\Lambda_{d,1}\Lambda_{d,2}\cdots \Lambda_{d,n}$ encodes all of the inequalities that must be satisfied in a $d$-fold diamond of length $n$. That is, each inequality (i.e., edge of the graph) corresponds to some $\lambda_{i,j}$ or $\mu_{i,j}.$ In other words, we immediately have that
\begin{equation} \label{eq:precrude}
D_{d,n} = \OO \sum_{a_i\geq 0, b_{j,k}\geq 0} \Lambda_{d,1}\Lambda_{d,2}\cdots \Lambda_{d,n} q_{0}^{a_0}\cdots q_{n}^{a_n} w^{\sum_{j,k} b_{j,k}}.
\end{equation}

In the proof of Theorem \ref{thm:mainthm}, we will also need (for integers $d, n\geq 1,$ and $\rho\geq 0$)
\[D_{d,n}^{(\rho)} \coloneqq D_{d,n}^{(\rho)}(q_0, q_1, \ldots, q_n;w) \coloneqq \sum q_{0}^{a_0}\cdots q_{n}^{a_n} w^{\sum_{j,k} b_{j,k}}\]
where the sum for $D_{d,n}^{(\rho)}$ ranges over all non-negative integers $a_i$ and $b_{j,k}$ which satisfy the inequalities encoded by $\Lambda_{d,1}\Lambda_{d,2}\cdots \Lambda_{d,n}$ and where $a_n\geq \rho.$

Now we set (for $d,n\geq 1$ and $1\leq k\leq n$)
\begin{align*}
h\coloneqq h_{d}&\coloneqq \frac{1}{1- \lambda_{1,1} \cdots \lambda_{1,d} \, q_{0}} \\ 
f_{k}\coloneqq f_{k,d} &\coloneqq \frac{1}{\left(1-\frac{\mu_{k,1}}{\lambda_{k,1}}w\right) \cdots \left(1-\frac{\mu_{k,d}}{\lambda_{k,d}}w\right) \cdot \left(1-\frac{\lambda_{k+1,1}\cdots \lambda_{k+1,d}}{\mu_{k,1}\cdots\mu_{k,d}} q_{k}\right)} \\
g_{n}\coloneqq g_{n,d} &\coloneqq \frac{1-\frac{\lambda_{n+1,1}\cdots \lambda_{n+1,d}}{\mu_{n,1}\cdots\mu_{n,d}} q_n}{1-\frac{q_n}{\mu_{n,1}\cdots\mu_{n,d}}}
\end{align*}
and note that one has the following crude form of the generating functions.
\begin{lem} \label{lem:prop2.1}
For $d, n\geq 1,$ we have that
\begin{align*}
D_{d,n} &= \OO h\cdot f_1 \cdots f_n \cdot g_n\\
D_{d,n}^{(\rho)} &= \OO h\cdot f_1 \cdots f_n \cdot g_n \left(\frac{q_n}{\mu_{n,1}\cdots \mu_{n,d}}\right)^\rho.
\end{align*}
\end{lem}
\begin{proof} The proof follows, mutatis mutandis, as in the proofs of Propositions 2.1 and 2.2 in \cite{AndrewsPauleRiese01b}.
\end{proof}

\subsection{Preliminary Results}
\label{sec:prelimresults}

In order to prove Theorem \ref{thm:mainthm}, we must be able to apply some elimination formulae in order to simplify the expression for $D_{d,n}$ given in Lemma \ref{lem:prop2.1}. We'll first do this in the case where $n=1.$

\begin{lem} \label{lem:basecase}
For $d\geq 1,$ we have that
\[D_{d,1}(q_0, q_1; w) = \frac{F_d(q_0,w)}{(1-q_0)(1-q_0w)\cdots (1-q_0w^d)\cdot (1-q_0q_1w^d)},\]
where $F_d(q_0,w)\in \ZZ[q_0,w]$ is a polynomial of degree $d-1$ in $q_0$ that is given by $F_1(q_0,w)=1$ and \[F_d(q_0,w) = \frac{(1-q_0w^d)F_{d-1}(q_0,w) - w(1-q_0)F_{d-1}(q_0w,w)}{1-w}.\]
\end{lem}

\begin{proof}
Proceed by induction on $d$. The case $d=1$ corresponds to classical partitions with at most three parts. We may compute the generating function for such partitions using MacMahon's partition analysis: we have 
\begin{align*} 
D_{1,1} (q_{0},q_{1};w) &= \OO \sum\limits_{a_{0},a_{1},b_{1} \geq 0} \lambda_{1}^{a_{0}-b_{1}} \mu_{1}^{b_{1}-a_{1}} q_{0}^{a_{0}} w^{b_{1}} q_{1}^{a_{1}} = \OO \sum\limits_{a_{0},a_{1},b_{1} \geq 0} (\lambda_{1}q_{0})^{a_{0}} \left(\frac{\mu_{1}}{\lambda_{1}} w\right)^{b_{1}} \left(\frac{q_{1}}{\mu_{1}}\right)^{a_{1}} \\
&= \OO \frac{1}{(1-\lambda_{1}q_{0})(1-\lambda_{1}^{-1}\mu_{1} w)(1-\mu_{1}^{-1} q_{1})}  = \OO \frac{1}{(1-q_{0})(1-\mu_{1}q_{0}w)(1-\mu_{1}^{-1}q_{1})} \\
&= \frac{1}{(1-q_{0})(1-q_{0}w)(1-q_{0}q_{1}w)},
\end{align*}
by applying \eqref{eq:MacMahon1} to successively eliminate $\lambda_{1}$ and $\mu_{1}.$ Hence, $F_{1}=1.$

Now suppose that the conclusion holds for $d-1.$ Then by Lemma \ref{lem:prop2.1} (and the induction hypothesis) we have
\begin{align*}
D_{d,1} (q_{0},q_{1};w) &= \OO \frac{1}{(1- \la_1 \cdots \la_d \, q_{0}) \cdot \left(1-\la_1^{-1}\mu_1w\right) \cdots \left(1-\la_d^{-1}\mu_dw\right) \cdot \left(1-\mu_1^{-1}\cdots\mu_d^{-1} q_1\right)}\\
&= \OO \frac{D_{d-1,1} (\lambda_{d}q_{0}, \mu_{d}^{-1} q_{1}; w)}{1-\lambda_{d}^{-1}\mu_{d}w}\\
&= \OO \frac{F_{d-1}(\lambda_{d}q_{0},w)}{(1-\lambda_{d}q_{0})\cdots(1-\lambda_{d}q_{0}w^{d-1})(1-\lambda_{d}\mu_d^{-1}q_{0}q_{1}w^{d-1})(1-\lambda_{d}^{-1}\mu_{d}w)}.
\end{align*}
Write $F_{d-1}$ as a polynomial in $q_{0}$ with coefficients in $\mathbb{Z}[w]$ as
\[ F_{d-1}(q_0,w) = \sum\limits_{i=0}^{d-2} a_{i}(w) q_{0}^{i}. \]
Then 
\begin{align*} 
D_{d,1}  (q_{0},q_{1};w) 
&= \OO \frac{\sum\limits_{i=0}^{d-2} a_{i}(w) (\lambda_{d}q_{0})^{i}}{(1-\lambda_{d}q_{0})\cdots(1-\lambda_{d}q_{0}w^{d-1})(1-\frac{\lambda_{d}}{\mu_{d}}q_{0}q_{1}w^{d-1})(1-\frac{\mu_{d}}{\lambda_{d}}w)} \\
&= \sum\limits_{i=0}^{d-2} a_{i}(w) q_{0}^{i} \OO \frac{\lambda_{d}^{i}}{(1-\lambda_{d}q_{0})\cdots(1-\lambda_{d}q_{0}w^{d-1})(1-\frac{\lambda_{d}}{\mu_{d}}q_{0}q_{1}w^{d-1})(1-\frac{\mu_{d}}{\lambda_{d}}w)} \\
&= \sum\limits_{i=0}^{d-2} a_{i}(w) q_{0}^{i} \OO \frac{\lambda_{d}^{i}}{(1-\lambda_{d}q_{0})\cdots(1-\lambda_{d}q_{0}w^{d-1})(1-\lambda_{d}^{-1}w)(1-q_{0}q_{1}w^{d})},
\end{align*}
by \eqref{eq:MacMahon1}. Rearranging, 
\begin{align*}
D_{d,1} (q_{0},q_{1};w) &= \frac{\sum\limits_{i=0}^{d-2} a_{i}(w) q_{0}^{i}}{1-q_{0}q_{1}w^{d}} \OO \frac{\lambda_{d}^{i}}{(1-\lambda_{d}q_{0})\cdots(1-\lambda_{d}q_{0}w^{d-1})(1-\lambda_{d}^{-1}w)} \\
&= \frac{\sum\limits_{i=0}^{d-2} a_{i}(w) q_{0}^{i}}{(1-q_{0}q_{1}w^{d})(1-w)} \left[ \frac{1}{(1-q_{0})\cdots(1-q_{0}w^{d-1})} - \frac{w^{i+1}}{(1-q_{0}w)\cdots(1-q_{0}w^{d})} \right] \\
&= \frac{1}{(1-q_{0})\cdots(1-q_{0}w^{d})(1-q_{0}q_{1}w^{d})} \sum\limits_{i=0}^{d-2} a_{i}(w) q_{0}^{i} \left[\frac{1-q_{0}w^{d}}{1-w} - \frac{w^{i+1}(1-q_{0})}{1-w} \right] \\
\end{align*}
by applying Proposition \ref{prop:MacMahon}. Thus, we have shown that 
\begin{align*}
F_{d}(q_{0},w) &= \sum\limits_{i=0}^{d-2} a_{i}(w) q_{0}^{i} \left[\frac{1-q_{0}w^{d}}{1-w} - \frac{w^{i+1}(1-q_{0})}{1-w} \right] \\
&= \frac{1-q_{0}w^{d}}{1-w} \sum\limits_{i=0}^{d-2} a_{i}(w)q_{0}^{i} - \frac{w}{1-w}(1-q_{0}) \sum\limits_{i=0}^{d-2} a_{i} (w) q_{0}^{i} w^{i} \\
&= \frac{(1-q_{0}w^{d}) F_{d-1}(q_{0},w) - w(1-q_{0})F_{d-1}(q_{0}w,w)}{1-w}.
\end{align*}
To complete the proof, note that $F_d(q_0,w)$ is indeed a polynomial since $w=1$ is a root of the numerator given above.
\end{proof}

Lemma \ref{lem:basecase} will serve as the base case of an induction argument in the proof of Theorem \ref{thm:mainthm}. We will also need the following results in order to complete that proof.

\begin{lem} \label{lem:lem21}
For $k\geq1,$ and $y_1, \ldots, y_k, z\neq 0$ distinct elements from an appropriate field, define
\[p(\mathbf{y};z)\coloneqq \prod_{i=1}^k\left(1-\frac{y_i}{z}\right),\]
where $\mathbf{y} = (y_1, \ldots, y_k).$ For $1\leq j\leq k$ define
\[p_j(\mathbf{y};z) \coloneqq \prod_{\substack{i=1\\i\neq j}}^k\left(1-\frac{y_i}{z}\right)^{-1}.\]
Then we have \[\frac{1}{p(\mathbf{y};z)} = \sum_{j=1}^k\frac{p_j(\mathbf{y};y_j)}{1-(y_j/z)}.\]
\end{lem}

\begin{proof}
See \cite[Lemma 2.1]{AndrewsPauleRiese01b}.
\end{proof}

\begin{lem} \label{lem:cor22}
Let $d,n\geq 1$ and $\rho\geq 0.$ Then
\[D_{d,n}^{(\rho)}(q_0,\ldots,q_n;w) = (q_0\cdots q_n)^\rho w^{dn\rho}D_{d,n}(q_0,\ldots,q_n;w).\]
\end{lem}

\begin{proof}
The above equality is given by the following bijection. Given an arbitrary $d$-fold diamond partition, adding $\rho$ to each part %(equivalent to multiplying the generating function by $(q_0\cdots q_n)^\rho w^{dn\rho}$)
assigns the partition a unique $d$-fold diamond partition with smallest part at least $\rho$. Conversely, given an arbitrary $d$-fold diamond partition with smallest part at least $\rho$, subtracting $\rho$ from each part assigns the partition a unique $d$-fold diamond partition. Thus set of $d$-fold diamond partitions and the set of $d$-fold diamond partitions with smallest part at least $\rho$ are in bijection, which is encoded by the above equality on their generating functions; this proves the result.
\end{proof}

\subsection{Proof of Theorem \ref{thm:mainthm}}
\label{sec:mainthm}

To prove Theorem \ref{thm:mainthm}, we follow the same approach as in the proof of Theorem 2.1 of \cite{AndrewsPauleRiese01b}.

\begin{proof} We proceed by induction on $n.$ The base case (for $n=1$) has already been discussed in Lemma \ref{lem:basecase}. Now, suppose the theorem holds for $n$ and note that by Lemma \ref{lem:prop2.1} we have

\begin{align*}
D_{n+1} = &D_{d,n+1}(q_0,\ldots,q_{n+1};w) = \OO h\cdot f_1\cdots f_{n+1} \cdot g_{n+1}\\
= &\OO h\cdot f_1\cdots f_{n-1} \frac{1}{\left(1-\frac{\mu_{n,1}}{\lambda_{n,1}}w\right) \cdots \left(1-\frac{\mu_{n,d}}{\lambda_{n,d}}w\right)} \cdot \frac{1}{\left(1-\frac{\lambda_{n+1,1}\cdots \lambda_{n+1,d}}{\mu_{n,1}\cdots\mu_{n,d}} q_{n}\right)}\\
&\cdot \frac{1}{\left(1-\frac{\mu_{n+1,1}}{\lambda_{n+1,1}}w\right) \cdots \left(1-\frac{\mu_{n+1,d}}{\lambda_{n+1,d}}w\right)} \cdot \frac{1}{\left(1-\frac{1}{\mu_{n+1,1}\cdots\mu_{n+1,d}} q_{n+1}\right)}.
\end{align*}

Now, we apply Lemma \ref{lem:basecase} to the last $d+2$ factors to eliminate $\la_{n+1,1}, \ldots \la_{n+1,d}$ and $\mu_{n+1,1}, \ldots \mu_{n+1,d}$ and find that
\begin{align*}
D_{n+1} = &\OO h\cdot f_1\cdots f_{n-1} \frac{1}{\left(1-\frac{\mu_{n,1}}{\lambda_{n,1}}w\right) \cdots \left(1-\frac{\mu_{n,d}}{\lambda_{n,d}}w\right)} D_1\left(\frac{q_n}{\mu_{n,1}\dots \mu_{n,d}}, q_{n+1};w\right)\\
= &\OO h\cdot f_1\cdots f_{n-1} \frac{1}{\left(1-\frac{\mu_{n,1}}{\lambda_{n,1}}w\right) \cdots \left(1-\frac{\mu_{n,d}}{\lambda_{n,d}}w\right)} \cdot \frac{F_d\left(\frac{q_n}{\mu_{n,1}\cdots \mu_{n,d}},w\right)}{p(\mathbf{y}; \mu_{n,1}\cdots \mu_{n,d})},
\end{align*}
where we set $\mathbf{y} = (y_1, \ldots, y_{d+2}) \coloneqq (q_n,q_nw,\dots, q_nw^d,q_nq_{n+1}w^d).$ Then, applying Lemma \ref{lem:lem21} and substituting $F_d\left(\frac{q_n}{\mu_{n,1}\cdots \mu_{n,d}}, w\right) = \sum_{i=0}^{d-1}a_i(w) \frac{q_n^i}{(\mu_{n,1}\cdots \mu_{n,d})^i}$ (where $a_i(w)\in \ZZ[w]$ depends on $d$) gives
\begin{align*}
D_{n+1} &= \OO h\cdot f_1\cdots f_{n-1} \frac{F_d\left(\frac{q_n}{\mu_{n,1}\cdots \mu_{n,d}},w\right)}{\left(1-\frac{\mu_{n,1}}{\lambda_{n,1}}w\right) \cdots \left(1-\frac{\mu_{n,d}}{\lambda_{n,d}}w\right)} \cdot \sum_{j=1}^{d+2} \frac{p_j(\mathbf{y}; y_j)}{1-\frac{y_j}{\mu_{n,1}\cdots \mu_{n,d}}}\\
&= \sum_{i=0}^{d-1} a_i(w) q_n^i \sum_{j=1}^{d+2} p_j(\mathbf{y}; y_j) \frac{1}{y_j^i}\left[\OO \frac{h\cdot f_1\cdots f_{n-1}}{\left(1-\frac{\mu_{n,1}}{\lambda_{n,1}}w\right) \cdots \left(1-\frac{\mu_{n,d}}{\lambda_{n,d}}w\right)} \cdot \frac{1}{1-\frac{y_j}{\mu_{n,1}\cdots \mu_{n,d}}} \cdot \left(\frac{y_j}{\mu_{n,1}\cdots \mu_{n,d}}\right)^i \right]\\
&= \sum_{i=0}^{d-1} a_i(w) q_n^i \sum_{j=1}^{d+2} p_j(\mathbf{y}; y_j) \frac{1}{y_j^i} D_n^{(i)}(q_0, q_1, \dots, q_{n-1}, y_j; w)\\
&= \sum_{i=0}^{d-1} a_i(w) Q_n^i w^{din} \sum_{j=1}^{d+2} p_j(\mathbf{y}; y_j) D_n(q_0, q_1, \dots, q_{n-1}, y_j; w)\\
&= F_d(Q_nw^{dn},w) \sum_{j=1}^{d+2} p_j(\mathbf{y}; y_j) D_n(q_0, q_1, \dots, q_{n-1}, y_j; w),
\end{align*}
where we have used Lemma \ref{lem:prop2.1} and Lemma \ref{lem:cor22}. By the induction hypothesis, we may observe that
\[D_n(q_0, q_1, \dots, q_{n-1}, y_j; w) = \frac{1-Q_nw^{dn}}{1-Q_{n-1}y_jw^{dn}}D_n(q_0, q_1, \dots, q_{n-1}, q_n; w),\]
and thus we have
\begin{align*}
D_{n+1} =& F_d(Q_nw^{dn},w) (1-Q_nw^{dn}) D_n(q_0, q_1, \dots, q_{n-1}, q_n; w) \sum_{j=1}^{d+2} \frac{p_j(\mathbf{y}; y_j)}{1-(q_0\dots q_{n-1})y_jw^{dn}}\\
=& F_d(Q_nw^{dn},w) (1-Q_nw^{dn}) D_n(q_0, q_1, \dots, q_{n-1}, q_n; w) \frac{1}{p(\mathbf{y}; \frac{q_n}{Q_nw^{dn}})}\\
=& F_d(Q_nw^{dn},w) (1-Q_nw^{dn}) D_n(q_0, q_1, \dots, q_{n-1}, q_n; w)\\
&\cdot \frac{1}{(1-Q_nw^{dn})(1-Q_nw^{dn+1})\cdots (1-Q_nw^{dn+d})(1-Q_{n+1}w^{dn+d})}
\end{align*}
by Lemma \ref{lem:lem21} with $z=(Q_{n-1}w^{dn})^{-1}.$ Thus we have shown that
\[D_{n+1} = \frac{F_d(Q_nw^{dn},w)}{(1-Q_nw^{dn+1})\cdots (1-Q_nw^{dn+d})(1-Q_{n+1}w^{dn+d})} D_n,\]
which gives the desired result.
\end{proof}

\subsection{Proof of Theorems \ref{thm:genfcnforall} and \ref{thm:genfcnforlinks}} \label{subsec:proofof1.1and1.2}

Now, Theorems \ref{thm:genfcnforall} and \ref{thm:genfcnforlinks} can be obtained as corollaries of Theorem \ref{thm:mainthm}.
\begin{proof}[Proof of Theorems \ref{thm:genfcnforall} and \ref{thm:genfcnforlinks}]
Theorem \ref{thm:genfcnforall} comes from applying Theorem \ref{thm:mainthm} to
\[\sum_{n=0}^\infty r_d(n) q^n= \lim_{n\rightarrow\infty} D_{d,n}(q,q,\ldots, q;q).\] To prove Theorem \ref{thm:genfcnforlinks}, we first apply Theorem \ref{thm:mainthm} to obtain
\[\sum_{n=0}^\infty s_d(n) q^n= \lim_{n\rightarrow\infty} D_{d,n}(q,q,\ldots, q;1) = \prod_{n=1}^\infty \frac{F_{d}(q^n,1)}{(1-q^{n})^{d+1}},\]
so we must show that $F_d(q,1)=A_d(q).$ Since $F_1(q,1) = A_1(q) = 1,$ it suffices to show that $F_d(q,1)$ and $A_d(q)$ satisfy the same recurrence \eqref{eulerian_poly_recur}.

Recalling the recurrence from Lemma \ref{lem:basecase} and writing $F_{d-1}(q,w)=\sum_{i=0}^{d-2}a_i(w)q^i$ gives 

\begin{align*}
F_d(q,w)(1-w) &= (1-qw^d)F_{d-1}(q,w) - w(1-q)F_{d-1}(qw,w)\\
&= (1-qw^d)\sum_{i=0}^{d-2}a_i(w)q^i - w(1-q)\sum_{i=0}^{d-2}a_i(w)q^iw^i\\
&= \sum_{i=0}^{d-2}a_i(w)q^i\left[1-qw^d - w^{i+1} + qw^{i+1} \right]\\
&= (1-w)\sum_{i=0}^{d-2}a_i(w)q^i\left[(1+w+\cdots +w^i) +q(w^{i+1} + \cdots + w^{d-1})) \right].\\
\end{align*}

Finally, dividing by $w-1$ and then setting $w=1$ gives
\begin{align*}
F_d(q,1) &= \sum_{i=0}^{d-2}a_i(1)q^i\left[(i+1) +q(d-i-1)) \right]\\
&= \sum_{i=0}^{d-2}a_i(1)q^i\left[1+(d-1)q +i(1-q) \right]\\
&= (1+(d-1)q)F_{d-1}(q, 1)+q(1-q)F'_{d-1}(q,1),
\end{align*}
as desired.
\end{proof}

\section{Ramanujan--Like Congruences Satisfied by $s_d(n)$}\label{sec:congproofs}
We begin this section by noting the following corollary of Theorem \ref{thm:genfcnforlinks}, which gives an alternate form of the generating function for $s_d(n)$.
\begin{cor} \label{cor:2.2}
We have
\begin{equation*}
\label{alt_gen_fn_s_d_ver2}
\sum_{n=0}^\infty s_d(n) q^n = \prod_{n=1}^\infty \left( \sum_{j=0}^\infty  (j+1)^d q^{jn} \right).  
\end{equation*}
\end{cor}

\begin{proof}
This follows immediately from Theorem \ref{thm:genfcnforlinks}, together with the following identity proved by Euler in \cite{Eul}.
$$
\sum_{j=0}^\infty  (j+1)^d q^{jn} = \frac{A_{d}(q^n)}{(1-q^{n})^{d+1}}.$$
\end{proof}

\begin{rem}
One can view Corollary \ref{cor:2.2} combinatorially by building an arbitrary Schmidt type $d$-fold partition diamond as follows. First, one may choose the nodes $a_0, a_1, \ldots$ by picking the difference, $j_n,$ between $a_{n-1}$ and $a_n$ for each positive integer $n$ (where $j_n=0$ for sufficiently large $n$). This completely determines all of the linking nodes $a_0 = \sum_{n\geq 1}j_n,$ $a_1=\sum_{n\geq 2}j_n,$ etc., and the $q^{j_nn}$ keeps track of their contribution. Then, there are $(j_n+1)^d$ options for the nodes $b_{n,1}, \ldots, b_{n,d}$ (since there are $j_n+1$ possibilities for each of those nodes).
\end{rem}

We utilize this new view of the generating function for $s_d(n)$ to easily prove a variety of arithmetic properties satisfied by these functions.  Prior to doing so, we remind the reader of two well--known results which will also be helpful below:  

\begin{lem}(Euler's Pentagonal Number Theorem)
\label{EulerPNT}
We have 
\begin{equation*}
\prod_{m=1}^\infty (1-q^m) = \sum_{n=-\infty}^{\infty} (-1)^n q^{n(3n+1)/2}.
\end{equation*}
\end{lem}

\begin{proof}
See \cite[Eq. (1.6.1)]{H}.    
\end{proof}
\begin{lem}(Jacobi) 
\label{Jacobi}
We have 
\begin{equation*}
\prod_{m=1}^\infty (1-q^m)^3 = \sum_{n=0}^{\infty} (-1)^n (2n+1) q^{n(n+1)/2}.
\end{equation*}
\end{lem}
\begin{proof}
See \cite[Eq. (1.7.1)]{H}.    
\end{proof}

We now prove the following theorem which yields an infinite family of Ramanujan--like congruences modulo arbitrarily large powers of 2.
\begin{thm}
\label{congs_mod_powers_2}
For all $d\geq 1$ and all $n\geq 0,$ 
$s_d(2n+1) \equiv 0 \pmod{2^d}.$  
\end{thm}
\begin{proof}
For fixed $d\geq 1,$ 
\begin{eqnarray*}
\sum_{n=0}^\infty s_d(n) q^n 
&=& 
\prod_{n=1}^\infty \left( 1 + 2^dq^n + 3^dq^{2n} +4^dq^{3n} + 5^dq^{4n} + 6^dq^{5n} + \dots  \right)  \\
&\equiv & 
\prod_{n=1}^\infty \left( 1 + 0q^n + 3^dq^{2n} +0q^{3n} + 5^dq^{4n} + 0q^{5n} + \dots  \right)  \pmod{2^d} \\
&=& 
\prod_{n=1}^\infty \left( 1 +  3^dq^{2n}  + 5^dq^{4n}  + \dots  \right).
\end{eqnarray*}
In the penultimate line above, we have used the fact that $2^d \, \vert \, (2j)^d$ for all $j\geq 1.$  The last expression above is a function of $q^2,$ and this immediately implies the result.  
\end{proof}
We note, in passing, that the $d=2$ case of Theorem \ref{congs_mod_powers_2} was proven by Andrews and Paule \cite[Corollary 2]{AndrewsPaule22}.

We next prove the following overarching lemma that will provide us with the machinery needed to prove infinite families of divisibility properties satisfied by $s_d(n).$  

\begin{lem}
\label{internal_congs_mod_p}
Let $k$ and $r$ be nonnegative integers and let $m\geq 2$.  For all $n\geq 0,$ $s_{\phi(m)k+r}(n) \equiv s_r(n) \pmod{m}$ where $\phi(m)$ is Euler's totient function.    
\end{lem}
\begin{proof}
Note that 
\begin{eqnarray*}
&& 
\sum_{n=0}^\infty s_{\phi(m)k+r}(n)q^n \\
&=& 
\prod_{n=1}^\infty \left( 1 + 2^{\phi(m)k+r}q^n + 3^{\phi(m)k+r}q^{2n} +4^{\phi(m)k+r}q^{3n} \right. \\
&& \ \ \ \ \ \ \ \ \ \  \left. {}+ 5^{\phi(m)k+r}q^{4n} + 6^{\phi(m)k+r}q^{5n} + \dots  \right)  \\
&=& 
\prod_{n=1}^\infty \left( 1 + (2^{\phi(m)})^k\cdot 2^r q^n + (3^{\phi(m)})^k\cdot 3^rq^{2n} +(4^{\phi(m)})^k\cdot 4^rq^{3n} \right. \\
&& \ \ \ \ \ \ \ \ \ \  \left. {}+ (5^{\phi(m)})^k\cdot 5^rq^{4n} + (6^{\phi(m)})^k\cdot 6^rq^{5n} + \dots  \right)  \\
&\equiv & 
\prod_{n=1}^\infty \left( 1 + 2^r q^n + 3^rq^{2n} +4^rq^{3n} + 5^rq^{4n} +  6^rq^{5n} + \dots  \right)  \pmod{m}\\
&=& 
\sum_{n=0}^\infty s_{r}(n)q^n.
\end{eqnarray*}
The penultimate line above follows from Euler's generalization of Fermat's Little Theorem.
\end{proof}

We now transition to a consideration of various families of congruences modulo 5.  We first prove an infinite family of such congruences thanks to Lemma \ref{internal_congs_mod_p} which uses the function $s_1(n)$ as its starting point.   
\begin{thm}
\label{4k1_mod5}
For all $k\geq 0$ and all $n\geq 0,$ 
$$s_{4k+1}(5n+2) \equiv s_{4k+1}(5n+3) \equiv s_{4k+1}(5n+4) \equiv 0 \pmod{5}.$$  
\end{thm}
\begin{proof}
Note that 
\begin{eqnarray*}
\sum_{n=0}^\infty s_1(n)q^n 
&=& 
\prod_{n=1}^\infty \frac{1}{(1-q^n)^2}  \\
&=& 
\prod_{n=1}^\infty \frac{(1-q^n)^3}{(1-q^n)^5} \\
&\equiv & 
\prod_{n=1}^\infty \frac{(1-q^n)^3}{(1-q^{5n})} \pmod{5} \\
&=& 
\left( \sum_{j=0}^\infty (-1)^j(2j+1)q^{j(j+1)/2} \right) \prod_{n=1}^\infty \frac{1}{(1-q^{5n})} 
\end{eqnarray*}
thanks to Lemma \ref{Jacobi}.  Since $\prod_{n=1}^\infty \frac{1}{(1-q^{5n})}$ is a function of $q^5$, the above product will satisfy a congruence modulo 5 in an arithmetic progression of the form $5n+r$, for $0\leq r\leq 4,$ if and only if such a congruence is satisfied by 
$$
\left( \sum_{j=0}^\infty (-1)^j(2j+1)q^{j(j+1)/2} \right). 
$$
Since $5n+2$ and $5n+4$ are never triangular numbers, this immediately tells us that, for all $n\geq 0,$
$$s_1(5n+2) \equiv s_1(5n+4) \equiv 0 \pmod{5}.$$
Moreover, we know that $5n+3 = j(j+1)/2$ if and only if $j\equiv 2 \pmod{5},$ and in such cases, $2j+1 \equiv 0 \pmod{5}.$  Because of the presence of the factor $2j+1$ in the sum above, we then see that, for all $n\geq 0,$  
$$s_1(5n+3) \equiv 0 \pmod{5}$$
since $2\cdot 2+1 = 5 \equiv 0 \pmod{5}.$
The remainder of the proof immediately follows from Lemma \ref{internal_congs_mod_p}.  
\end{proof}
We note, in passing, that Baruah and Sarmah \cite[Theorem 5.1, equation (5.4)]{BS} also provided a proof of the $k=0$ case of Theorem \ref{4k1_mod5}.

We next consider the following modulo 5 congruences, with a focus on the case $d=2.$
\begin{thm}
\label{d2_25_23}
For all $k\geq 0$ and all $n\geq 0,$ 
$$s_{4k+2}(25n+23) \equiv 0 \pmod{5}.$$   
\end{thm}
\begin{proof}
We begin by noting that 
\begin{eqnarray*}
\sum_{n=0}^\infty s_2(n) q^n 
&=& 
\prod_{n=1}^\infty \frac{1+q^n}{(1-q^n)^3}  \\
&=& 
\prod_{n=1}^\infty \frac{1-q^{2n}}{(1-q^n)^4}  \\
&=& 
\prod_{n=1}^\infty \frac{(1-q^{2n})(1-q^n)}{(1-q^n)^5}  \\
&\equiv & 
\prod_{n=1}^\infty \frac{(1-q^{2n})(1-q^n)}{(1-q^{5n})} \pmod{5} \\
&=& 
\left( \sum_{j, k=-\infty}^\infty (-1)^{j+k}q^{2(j(3j+1)/2) + k(3k+1)/2} \right) \prod_{n=1}^\infty \frac{1}{(1-q^{5n})} 
\end{eqnarray*}
thanks to Lemma \ref{EulerPNT}.  
At this point, we consider those terms of the form $q^{25n+23}$ in the power series representation of the last expression above. There are no terms of the form $q^{25n+8}, q^{25n+13}, q^{25n+18},$ or  $q^{25n+23}$ in the double sum above, although there are terms of the form $q^{25n+3}.$ Thus, once we multiply the double sum above by 
\begin{equation}
\label{genfn_p5n}
    \prod_{n=1}^\infty \frac{1}{(1-q^{5n})} = \sum_{n=0}^\infty p(n)q^{5n},
\end{equation} 
we see that the only way to obtain a term of the form $q^{25n+23}$ which can contribute to $s_2(25n+23)$ modulo 5 is to multiply by a term of the form $q^{5(5m +4)}$ from the power series representation of (\ref{genfn_p5n}).  This will then contribute, modulo 5, to the value of $s_2(25n+23)$ by multiplying by the value $p(5m +4).$  Thanks to Ramanujan's well--known result that, for all $m \geq 0,$ $p(5m +4)\equiv 0 \pmod{5},$ we then know that, for all $n\geq 0,$ $s_{2}(25n+23) \equiv 0 \pmod{5}.$  The theorem then follows thanks to Lemma \ref{internal_congs_mod_p}.      
\end{proof}

We can also show that the functions $s_{4k+3}$ satisfy a rich set of congruences modulo 5 via the following:

\begin{thm}
\label{4k3_mod5_5n}
For all $k\geq 0$ and all $n\geq 0,$ 
$$s_{4k+3}(5n+2) \equiv s_{4k+3}(5n+4) \equiv 0 \pmod{5}.$$  
\end{thm}
\begin{proof}
We begin with the following generating function manipulations:  
\begin{eqnarray*}
\sum_{n=0}^\infty s_3(n) q^n 
&=& 
\prod_{n=1}^\infty \frac{1+4q^n+q^{2n}}{(1-q^n)^4}  \\  
&\equiv & 
\prod_{n=1}^\infty \frac{1-q^n+q^{2n}}{(1-q^n)^4}  \pmod{5} \\  
&=& 
\prod_{n=1}^\infty \frac{1+(-q^n)+(-q^n)^2}{(1-q^n)^4}  \\  
&=& 
\prod_{n=1}^\infty \frac{1-(-q^n)^3}{(1-(-q^n))(1-q^n)^4}  \\  
&=& 
\prod_{n=1}^\infty \frac{1+q^{3n}}{(1+q^n)(1-q^n)^4}  \\ 
&=& 
\prod_{n=1}^\infty \frac{(1-q^{6n})(1-q^n)^2}{(1-q^{3n})(1-q^{2n})(1-q^n)^5} \\ 
&\equiv & 
\prod_{n=1}^\infty \frac{(1-q^{6n})(1-q^n)^2}{(1-q^{3n})(1-q^{2n})}\prod_{n=1}^\infty \frac{1}{(1-q^{5n})} \pmod{5} 
\end{eqnarray*}
We now consider the power series representation of  
\begin{equation}
\label{modform7}
F(q):=\prod_{n=1}^\infty \frac{(1-q^{6n})(1-q^n)^2}{(1-q^{3n})(1-q^{2n})},
\end{equation}
which turns out to be a well--known modular form; indeed, it appears as the seventh modular form in Mersmann's list of the 14 primitive eta--products which
are holomorphic modular forms of weight 1/2.  See \cite[page 30]{123MF} and \cite[Entry A089812]{OEIS} for additional details.  After some elementary calculations (for example, by noting that $qF(q^8)$ appears in \cite[Theorem 1.2]{LemkeOliver13}), we find that 
$$
F(q)= \sum_{t=0}^\infty q^{t(t+1)/2} - 3\sum_{t=0}^\infty q^{(3t+1)(3t+2)/2}.
$$
As noted in the proof of Theorem \ref{4k1_mod5}, $5n+2$ and $5n+4$ can never be triangular numbers. Since $\frac{1}{2}(3t+1)(3t+2)$ is always triangular, we conclude that $s_3(5n+2) \equiv 0 \pmod{5}$ and $s_3(5n+4) \equiv 0 \pmod{5}$. The case where $k > 0$ follows from Lemma \ref{internal_congs_mod_p}.
\end{proof} 
Related to the above, we can also prove the following additional congruence family modulo 5.  
\begin{thm}
\label{d3_25_23}
For all $k\geq 0$ and $n\geq 0,$ 
$$s_{4k+3}(25n+23)  \equiv 0 \pmod{5}.$$  
\end{thm}
\begin{proof}
As with the proof of Theorem \ref{4k3_mod5_5n} above, we begin by considering $s_3(n)$.  Using the notation from above, recall that 
\begin{eqnarray}
\sum_{n=0}^\infty s_3(n) q^n 
&\equiv & 
\left( \sum_{t=0}^\infty q^{t(t+1)/2} - 3\sum_{t=0}^\infty q^{(3t+1)(3t+2)/2} \right) \prod_{n=1}^\infty \frac{1}{(1-q^{5n})} \pmod{5} \notag \\
&=& 
\left( \sum_{t=0}^\infty q^{t(t+1)/2} - 3\sum_{t=0}^\infty q^{(3t+1)(3t+2)/2} \right) \left( \sum_{n=0}^\infty p(n)q^{5n} \right).  \label{eq2523}
\end{eqnarray}
We then consider ways that $q^{25n+23}$ can arise when (\ref{eq2523}) is expanded.  We note that no numbers of the form $25n+8, 25n+13, 25n+18,$ or $25n+23$ can be represented as a triangular number.  Thus, the only way to obtain a term of the form $q^{25n+23}$ in (\ref{eq2523}) is to multiply by a term of the form $p(5m+4)q^{5(5m+4)}$ (in the same way as was discussed in the proof of Theorem \ref{d2_25_23}).  This implies that, for all $n\geq 0,$ $s_{3}(25n+23)  \equiv 0 \pmod{5}.$  The full result follows from Lemma \ref{internal_congs_mod_p}.
\end{proof}

We close this section with one last infinite family of congruences, this time modulo 11.  
\begin{thm}
\label{mod11_cong}
For all $n\geq 0,$ 
$$
s_{10k+1}(121n+111) \equiv 0 \pmod{11}.
$$
\end{thm}
\begin{proof}
In the work of Gordon \cite[Theorem 2]{Gor}, we set $k=2$ and $r=2,$ so that $\alpha=1.$  We then see that, if $24n\equiv 2\pmod{11^2},$ then $12n\equiv 1 \pmod{121}$ or $n\equiv 111\pmod{121}$.  Gordon's work then implies that, for all $n\geq 0,$ $s_1(121n+111) \pmod{11}$.  
The full result then follows from Lemma \ref{internal_congs_mod_p}.

\end{proof}

%%%%%%%%%%%%%%%%%%%%%%%%%%%%%%%%%%%%%%%%%%%%

\section{Closing Thoughts}

We close this work with the following set of conjectured congruence families modulo 7:  
\begin{conjecture}
For all $k\geq 0$ and $n\geq 0,$ 
\begin{eqnarray*}
s_{6k+1}(49n+17) &\equiv & s_{6k+2}(49n+17) \ \equiv \ 0 \pmod{7}, \\
s_{6k+1}(49n+31) &\equiv & s_{6k+2}(49n+31) \ \equiv \  0 \pmod{7}, \\
s_{6k+1}(49n+38) &\equiv & s_{6k+2}(49n+38) \ \equiv \  0 \pmod{7}, \\
s_{6k+1}(49n+45) &\equiv & s_{6k+2}(49n+45) \ \equiv \  0 \pmod{7}.
\end{eqnarray*}
\end{conjecture}
We would be very interested to see elementary proofs of the above results. \\

In addition, Theorems \ref{congs_mod_powers_2} and \ref{4k3_mod5_5n} provide three Ramanujan congruences for $s_{3}(n),$ namely 
\begin{align*}
s_{3}(2n+1) &\equiv 0 \pmod 2 \\
s_{3}(5n+2) &\equiv 0 \pmod 5 \\
s_{3}(5n+4) &\equiv 0 \pmod 5,
\end{align*}
by taking $d=3$ and $k=0,$ respectively. We would like to know if there are any other Ramanujan congruences for $s_{3}(n)$ and if any exist for $d \geq 4.$ We are especially interested in knowing if $s_{d}(n)$ satisfies only finitely many Ramanujan congruences for fixed $d \geq 3.$ 

\section{Acknowledgements}
The authors would like to thank the referees for their valuable comments and suggestions, and Larry Rolen for helpful conversations about this work.

\bibliographystyle{alpha} 
\bibliography{refs}

\end{document}